\documentclass[]{elsarticle}

\usepackage{amssymb}
\usepackage{amsmath,amsthm,amssymb}

\newcommand{\R}{{\mathbb R}}

\newcommand{\dsum}{\displaystyle\sum}
\newcommand{\dlim}{\displaystyle\lim}

\DeclareMathOperator{\pV}{pV}

\newtheorem{theorem}{Theorem}[section]
\newtheorem{lemma}[theorem]{Lemma}
\newdefinition{example}[theorem]{Example}
\newtheorem{corollary}[theorem]{Corollary}
\newtheorem{proposition}[theorem]{Proposition}
\newdefinition{remark}[theorem]{Remark}
\newproof{prf}{Proof}
\numberwithin{equation}{section}

\newcommand{\sol}{\renewcommand{\proofname}{\it Solution}\begin{proof}}
\newcommand{\finsol}{\end{proof}}
\newcommand{\ope}[1]{\mathop{\mathrm{#1}}}

\def\b{\beta}

\def\s{\sigma}\def\de{\delta}

\newcommand{\field}[1]{\mathbb{#1}}
\def\R{\field{R}}\def\Z{\field{Z}}

\DeclareMathSymbol{\smallsetminus}  {\mathbin}{AMSb}{"72}
\DeclareMathSymbol{\pitchfork}    {\mathrel}{AMSa}{"74}
\DeclareMathSymbol{\preccurlyeq}  {\mathrel}{AMSa}{"34}
\DeclareMathSymbol{\lneqq}        {\mathrel}{AMSb}{"08}
\DeclareMathSymbol{\gneqq}        {\mathrel}{AMSb}{"09}
\DeclareMathSymbol{\nleq}         {\mathrel}{AMSb}{"02}
\DeclareMathSymbol{\twoheadrightarrow}  {\mathrel}{AMSa}{"10}
\DeclareMathSymbol{\subsetneq}      {\mathrel}{AMSb}{"28}
\DeclareMathSymbol{\lesseqqgtr}   {\mathrel}{AMSa}{"53}
\DeclareMathSymbol{\square}       {\mathord}{AMSa}{"03}
\DeclareMathSymbol{\gtrless}      {\mathrel}{AMSa}{"3F}
\usepackage{amsmath}
\usepackage{amscd}
\usepackage{verbatim}
\usepackage{amssymb}

\usepackage{epsfig}
\usepackage{epsf}

\begin{document}
\begin{frontmatter}
\title{Euler-Maclaurin formulas for functions of bounded variation}
\author[MUNIPD]{Giuseppe De Marco}
\ead{gdemarco@math.unipd.it}
\author{Marco De Zotti}
\ead{marco.dezotti@studenti.unipd.it}
\author[MUNIPD]{Carlo Mariconda}
\ead{carlo.mariconda@unipd.it}
\address[MUNIPD]{Dipartimento di Matematica, Universit\`a
degli Studi di Padova, Via Trieste 63, 35121 Padova, Italy}

\begin{abstract}
The first-order  Euler-Maclaurin formula relates the sum of the values of a smooth function on an interval of integers with its integral on the same interval on $\mathbb R$.
We formulate here the analogue  for functions that are just of bounded variation.
\end{abstract}
\end{frontmatter}
\section*{Notation}
Our main reference for the basic facts and related notation  on BV functions is \cite{FP}.
Let us recall that a real valued function $f$ defined on an interval $I$  is of Bounded Variation (we often simply write BV) if the so-called \emph{pointwise variation} $\pV(f,I)$ of $f$ on $I$, given by
\[\pV(f,I):=\sup\left\{\sum_{0\le i<n}|f(t_{i+1})-f(t_i)|:\, t_i\in I, \, t_0<t_1<\cdots <t_n\right\}\]
is finite.
In this case  there exist two increasing and bounded functions $f_1, f_2:I\to \mathbb{R}$ satisfying
\begin{equation}\label{tag:bv}
f=f_1-f_2,\qquad \pV(f,I)=\pV(f_1,I)+\pV(f_2,I).
\end{equation}
In particular, every function of bounded variation is locally integrable.
The left and right limit of a BV function $f$ in $c$ will be denoted, respectively, $f(c^-)$ and $f(c^+)$.

We find useful here to adopt the following sum notation that is quite common in the field of Discrete Calculus: if $a<b$ are natural numbers we set
\[\dsum_{a\le k<b}f(k):=\dsum_{k=a}^{b-1}f(k).\]

Moreover, we set $\left[f\right]_a^b=f(b)-f(a)$.
\section{Introduction}
The first order Euler-Maclaurin formula for a smooth function $f:[a,b]\to\mathbb R$ ($a<b$ in $\Z$) states that
\begin{equation}\label{tag:em1}
\sum_{a\le k<b}f(k)=\int_a^bf(t)\,dt-\dfrac12\left[f\right]_a^b+R,\quad R=\int_a^bf'(t)B_1(t-[t])\,dt,
\end{equation}
where $B_1(t)=t-\dfrac12$ is the first Bernoulli polynomial. The formula is useful in the approximation of finite sums, and to relate the convergence of generalized integrals with that of numerical series:
we refer to \cite{MT, KP, L} for a survey on the subject.
Since $|B_1|\le\dfrac12$ on $[0,1]$ it follows that the remainder $R$ is bounded above by $\displaystyle\dfrac12\int_a^b|f'(t)|\,dt$, so that if $f$ is monotonic one has \[|R|\le \dfrac12|f(b)-f(a)|.\]
The proof is based on a simple, though smart,  integration by parts and begins assuming that $f$ is defined $[0,1]$: since $B_1'=1$, writing that
\[\int_0^1f(t)\,dt=\int_0^1f(t)B_1'(t)\,dt=[fB_1]_0^1-\int_0^1f'(t)B_1(t)\,dt\]
yields
\[f(0)=\int_0^1f(t)\,dt-\dfrac 12\left[f\right]_0^1+\int_0^1f'(t)B_1(t)\,dt,\]
which is (\ref{tag:em1}) when $a=0$ and $b=1$.

In Theorem~\ref{remark:critint4} we show that if $f$ is just of bounded variation (BV) on $[a,b]$ then (\ref{tag:em1}) holds with the exception that the remainder $R$ is bounded above by $\dfrac12\pV(f, [a,b])$. The proof of the result is elementary: indeed one can deal  with monotonic function, and adapt the  same arguments that are involved in the proof of  the  integral criterion for the convergence of a series with monotonic terms; part of the material arises from the  thesis \cite{DZ}.
In the final part of Section~\ref{subsect:monotonic} we obtain the results that follow the traditional Euler-Maclaurin formula for a smooth function, that is assumed here to be just BV: the approximation of the partial sums of the series $\dsum_{k=0}^Nf(k)$ in terms of $\dsum_{k=0}^nf(k)$ ($n<N$), the existence of the Euler constant with a related  asymptotic formula for $\dsum_{k=0}^nf(k)$ as $n\to +\infty$ and a generalization to BV functions of the integral test for the convergence of a series.

In Section~\ref{sec:3} we prove a version of \eqref{tag:em1} based on a partial integration formula for BV functions; in this formula the measure theoretic variation of the function is involved, which may be smaller than the point variation for discontinuous functions, and to deduce \eqref{tag:em1} from it we need to explicit the formula that connects the two variations; this is done in Proposition~\ref{prop:pVlambda}.

We are not aware of other formulations of the Euler-Maclaurin formulas for BV functions in the spirit of Theorem~\ref{remark:critint4}. Instead, the approximation formula for the sum  of a series (Corollary~\ref{es:criterio integrale}) was established in a more general setting in \cite{TrigubP}, \cite[4.1.5]{TrigubB} for functions whose $r$-th derivative is BV. The methods involved there arise from Fourier analysis, far from our elementary approach. A recent extension, comparing  in the multidimensional case the Fourier integral of a function of bounded variation and the corresponding trigonometric series with its Fourier coefficients
was recently established in \cite{Lifl}: we thank Elijah Liflyand for sharing the above results with us.

\section{A Euler-Maclaurin type formula for BV functions and its consequences}\label{subsect:monotonic}
\subsection{A Euler-Maclaurin type formula}
\begin{theorem}[Euler-Maclaurin type formula for BV functions]\label{remark:critint4}
Let $a,b$ in $\mathbb Z$ and $f:[a,b]\to\mathbb R$ be a function of bounded variation. Then
\begin{equation}\label{tag:critint_new}\begin{aligned}
\sum_{a\le k<b}f(k)&=\int_a^bf(x)\,dx-\dfrac12\left[f\,\right]_a^b+R,
&|R|\le \dfrac12\pV(f,[a,b]).\end{aligned}\end{equation}
\end{theorem}
\begin{proof}
Assume first that $f$ is monotonic increasing.
On every interval $[k, k+1]$ ($k\in\mathbb Z$) contained in $[a,b]$ one has
\[f(k)\le f(t)\le f(k+1)\qquad \forall t\in [k, k+1],\]
from which it follows that
\[f(k)=\int_k^{k+1}f(k)\,dt\le \int_k^{k+1}f(t)\,dt\le \int_k^{k+1}f(k+1)\,dt=f(k+1).\]
Summing the terms of the foregoing inequalities, as $k$ varies between $a$ and $b-1$, one obtains
\[\dsum_{a\le k<b}f(k)\le \int_a^bf(t)\,dt\le \dsum_{a\le k<b}f(k)+f(b)-f(a).\]
Subtracting the term $\dfrac 12(f(b)-f(a))$ from the  members of the preceding inequalities one finds
\[\begin{aligned}\dsum_{a\le k<b}f(k)-\dfrac 12(f(b)-f(a))&\le \int_a^bf(t)\,dt-\dfrac 12(f(b)-f(a))\\
&\le \dsum_{a\le k<b}f(k)+\dfrac 12(f(b)-f(a)),\end{aligned}\]
from which the conclusion follows.

\noindent
If $f$ is of bounded variation, let $f_1,f_2$ be as in (\ref{tag:bv}): since
\[\dsum_{a\le k<b}f_i(k)=\int_a^bf_i(x)\,dx-\dfrac12\left[f_i\,\right]_a^b+R_i, \quad |R_i|\le \dfrac12\pV(f_i,[a,b]) \quad (i=1,2)\]
by subtracting term by term we get
\[\dsum_{a\le k<b}f(k)=\int_a^bf(x)\,dx-\dfrac12\left[f\,\right]_a^b+R,\qquad R=R_1-R_2,\]
so that
\[|R|\le |R_1|+|R_2|\le \dfrac12\pV(f_1,[a,b])+\dfrac12 \pV(f_2,[a,b])=\dfrac12\pV(f,[a,b]).\]
\end{proof}
\begin{remark}\label{rem:monotonic} If $f$ is monotonic on $[a,b]$ then $\pV(f,[a,b])=|f(b)-f(a)|$, the remainder term $R$ can be thus estimated by $\dfrac12|f(b)-f(a)|$: this fact is well known as a consequence of the Euler-Maclaurin formula when $f$ is monotonic or of class $C^1$ \cite{MT}.
\end{remark}
\begin{corollary}[The approximation formula for finite sums]\label{coro:stimasummonotonic}
Let $f:[0,+\infty[\to\mathbb R$ be of bounded variation. For every $N\ge n$ the following \textbf{approximation formula} holds:

\begin{equation}\label{tag:stima_ridotte_monotone}\begin{aligned}
\dsum_{0\le k<N}f(k)&=\dsum_{0\le k<n}f(k)+\int_n^Nf(x)\,dx-\dfrac12\left[f\,\right]_n^N+\varepsilon_1(n,N),\\
&|\varepsilon_1(n,N)|\le \dfrac12\pV(f,[n,N])\le \dfrac12\pV(f,[n, +\infty[).\end{aligned}\end{equation}

\end{corollary}
\begin{proof} It is enough to remark that
\[\dsum_{0\le k<N}f(k)-\dsum_{0\le k<n}f(k)=\dsum_{n\le k<N}f(k)\]
and to apply (\ref{tag:critint_new}) with  $a=n$ and $b=N$.
\end{proof}

\subsection{A generalization of the integral criterion for the convergence of a series}
Let $f:[0, +\infty[\to\mathbb R$ be locally integrable. We set
\[\gamma^f_n\index{$\gamma^f_n$}:=\sum_{0\le k<n}f(k)-\int_0^nf(x)\,dx\qquad\forall n\in\mathbb N.\]
Notice that, if $f$ is of bounded variation, then $f(\infty):=\displaystyle\lim_{x\to +\infty}f(x)$ exists and is finite.
\begin{theorem}[The Euler constant]\label{Maclaurin1monotasint}
Let $f:[0, +\infty[\to \mathbb R$ be of bounded variation.
The Euler constant of $f$ defined by
$\gamma^f:=\displaystyle\lim_{n\to +\infty}\gamma^f_n$ exists and is finite, and the  following \textbf{estimate}\index{Euler!costant!approssimation!monotonic functions} of $\gamma^f$ holds:

\begin{equation}\label{tag:approx_eulero_monotonic}
\gamma^f=\gamma^f_n-\dfrac12\left[f\,\right]_n^{\infty}+\varepsilon_1(n),\quad
|\varepsilon_1(n)| \le \dfrac{1}{2}\pV(f,[n, +\infty[)\qquad \forall n\in\mathbb N.
\end{equation}

\end{theorem}
\begin{proof}
Given $n,N\in\mathbb N$ with $N>n$, by Theorem~\ref{remark:critint4} we have
\begin{equation}\label{tag:bepi}
\gamma_N^f-\gamma_n^f=\dsum_{n\le k<N}f(k)-\int_n^Nf(x)\,dx=-\dfrac12\left[f\,\right]_n^N+R(n,N),
\end{equation}
with $|R(n,N)|\le \dfrac12\pV(f,[n,N])$.

\noindent
Since the limits $\dlim_{N\to +\infty}f(N)$ and $\dlim_{N\to +\infty}\pV(f,[0,N])=\pV(f,[0, +\infty[)$ are both finite, and
$\pV(f,[n,N])=\pV(f,[0,N])-\pV(f,[0,n])$, it follows from   the necessary part of the Cauchy convergence criterion  that
\[\dlim_{n,N\to +\infty}-\dfrac12\left[f\,\right]_n^N+R(n,N)=0.\]
The sufficiency part of the very same criterion thus implies that the limit $\dlim_{n\to +\infty}\gamma^f_n$ exists and is finite. Passing to the limit in (\ref{tag:bepi}) we get
\[\gamma^f-\gamma_n^f=\dsum_{n\le k<N}f(k)-\int_n^Nf(x)\,dx=\dfrac12(f(\infty)-f(n))+\varepsilon_1(n),\]
where $\varepsilon_1(n):=\dlim_{N\to +\infty}R(n,N)$ is dominated by $\dfrac12\pV(f,[n,+\infty[)$.
\end{proof}
An immediate consequence of Theorem~\ref{Maclaurin1monotasint} is the following generalization of the well known integral criterion for the convergence of the series $\dsum_{k=0}^{\infty}f(k)$ for bounded and monotonic functions.
\begin{corollary}[Integral criterion for series and approximation of its sum]\label{es:criterio integrale}\index{criterion!integral!series with monotonic terms} Let $f:[0,+\infty[\to \mathbb R$ be of bounded variation.
\begin{enumerate}
\item
The series
$\dsum_{k=0}^{\infty}f(k)$ and the generalized integral $\displaystyle\int_0^{+\infty}f(x)\,dx$ have the same behavior: both are either convergent or divergent.
\item Assume that the series
$\dsum_{k=0}^{\infty}f(k)$ converges.
For every $n\in\mathbb N$ the following \textbf{approximation }\index{series!monotonic terms!sum!approximation} holds:

\begin{equation}\label{tag:stima_ridotte_monotone_convergenti}\begin{aligned}
\dsum_{k=0}^{\infty}f(k)&=\dsum_{0\le k<n}f(k)+\int_n^{+\infty}f(x)\,dx-\dfrac12\left[f\,\right]_n^{\infty}+\varepsilon_1(n),\\
&|\varepsilon_1(n)|\le \dfrac12\pV(f,[n,+\infty[).\end{aligned}\end{equation}
\end{enumerate}
\end{corollary}
\begin{proof}
1. We know from Theorem~\ref{Maclaurin1monotasint} that
\[\gamma^f=\lim_{n\to\infty}\left(\dsum_{0\le k<n}f(k)-\int_0^n\!\!f(x)\,dx\right)\in\mathbb R.\]
Thus $\dsum_{k=0}^{\infty}f(k)$ and the limit $\displaystyle\lim_{\substack{n\to +\infty\\n\in\mathbb N}}\int_0^{n}f(x)\,dx$ have the same behavior. Since $f(\infty)$ belongs to $\mathbb R$, the value of $\displaystyle\lim_{\substack{n\to +\infty\\n\in\mathbb N}}\int_0^{n}f(x)\,dx$ coincides with that of $\displaystyle\int_0^{+\infty}f(x)\,dx$: the conclusion follows.

\noindent
2.
It follows from (\ref{tag:stima_ridotte_monotone}) that for every $N\ge n$ we have
\begin{equation}\label{tag:woihfeo}\dsum_{0\le k<N}f(k)=\dsum_{0\le k<n}f(k)+\int_n^Nf(x)\,dx-\dfrac12\left[f\,\right]_n^N+\varepsilon_1(n,N),\end{equation}
with
$|\varepsilon_1(n,N)|\le \dfrac12\pV(f,[n,N])\le  \dfrac12\pV(f,[n,+\infty[)$.
From Point 1. we know that $f$ is integrable in a generalized sense on $[0, +\infty[$. Passing to the limit for $N\to +\infty$ in (\ref{tag:woihfeo}) we deduce that $\varepsilon_1(n):=\displaystyle\lim_{N\to +\infty}\varepsilon_1(n,N)$ is finite, whence the validity of  (\ref{tag:stima_ridotte_monotone_convergenti}).
\end{proof}
\begin{remark} The approximation formula \eqref{tag:stima_ridotte_monotone_convergenti} was established, for a wider class of functions and with an explicit form of the reminder, in \cite{TrigubP}, \cite[4.1.5]{TrigubB} by means of Fourier analysis methods.
\end{remark}
\subsection{Asymptotic formulas}
\begin{theorem}[Asymptotic formulas]\label{coro:stimaasint_ordine1_monotona}

Let $f:[0, +\infty[\to \mathbb R$ be a function.
\begin{enumerate}
 \item If $f$ is of bounded variation,
 then for every $n\in\mathbb N$
\[
\label{tag:asymptotic_monotoinic3435}
\dsum_{0\le k<n}f(k)=\gamma^f+\int_0^nf(x)\,dx+\varepsilon'_1(n),\qquad
|\varepsilon'_1(n)|\le \pV(f,[n,+\infty[).
\]
\item If $f$ is monotonic and unbounded then for every $n\in\mathbb N$ we have

\[
\dsum_{0\le k<n}f(k)=\int_0^nf(x)\,dx+O\left(f(n)\right)\quad n\to +\infty;
\]
\end{enumerate}
\end{theorem}
\begin{proof}1.
From (\ref{tag:approx_eulero_monotonic}) we obtain
\[\gamma^f_n=\gamma^f+\varepsilon'_1(n),\]
where $\varepsilon'_1(n):=\dfrac12\left[f\,\right]_n^{\infty}-\varepsilon_1(n)$
and since $|\varepsilon_1(n)| \le \dfrac{1}{2}\pV(f,[n, +\infty[)$,  the following estimate holds
\[|\varepsilon'_1(n)|=\left|\dfrac12\left[f\,\right]_n^{\infty}-\varepsilon_1(n)\right|\le \pV(f,[n, +\infty[):\]
the conclusion follows.

\noindent
2. It follows from Theorem~\ref{remark:critint4}, together with Remark~\ref{rem:monotonic}, that for every $n\in\mathbb N$
\[\dsum_{0\le k<n}f(k)=\int_0^nf(x)\,dx-\dfrac12(f(n)-f(0))+ R(n),\]
with $|R(n)|\le \dfrac12|f(n)-f(0)|$.
Since $\displaystyle\lim_{n\to +\infty}f(n)=\pm\infty$, then
\[f(n)-f(0)= O(f(n))\qquad n\to +\infty,\]
whence $-\dfrac12(f(n)-f(0))+ R(n)=O(f(n))$ for $n\to +\infty$: the conclusion follows.
\end{proof}
\section{The Euler-Maclaurin formula for BV functions: a more measure theoretic look}\label{sec:3}
\subsection{Variation and point variation}
A function of locally bounded variation (i.e. of bounded variation on every bounded interval) $f:\R\to\R$ provides a finite signed measure $\mu_f$ on the $\s-$algebra of Borel subsets
of any  subinterval of $\R$ on which $f$ is bounded, in particular on any bounded interval. Denoting by $f(x^-)$ (resp. $f(x^+)$) the left (resp. right) limit of $f$ at a point $x$, the measures of  bounded  intervals with end-points $c<d$  are:
 \[\mu_f\big(]c,d[\big)=f(d^-)-f(c^+),\, \mu_f\big([c,d]\big)=f(d^+)-f(c^-),\]
 \[\mu_f\big([c,d[\big)=f(d^-)-f(c^-),\,\mu_f\big(]c,d]\big)=f(d^+)-f(c^+),\]
and for $c=d$ we have  $\mu_f\big(\{c\}\big)=f(c^+)-f(c^-)$, the jump of $f$ at $c$.
As for every signed measure  the {\em total variation measure} $|\mu_f|$ of the Borel set $E$ is
\[|\mu_f|(E)=\sup\left\{\sum_{k=1}^m|\mu_f(A_k)|: \,A_1,\dots,A_m\subseteq E \text{ disjoint and Borel}\right\}.\]
When $E$ is an   interval one can prove that the same supremum is obtained if  $A_1,\dots,A_m$ range only over subintervals of $E$, so that, if  $E$ is an interval
\begin{multline*}|\mu_f|(E)\!=\!\sup\left\{\sum_{k=1}^m|\mu_f\big(]x_{k-1}, x_k[\big)|+\sum_{k=0}^m|\mu_f\big(\{x_k\}\big)|:\,x_k\in E,\,x_0<\dots<x_m\right\}\\
=\sup\left\{\sum_{k=1}^m|f(x_k^-)-f(x_{k-1}^+)|+\sum_{k=0}^m|f(x_k^+)-f(x_k^-)|:\,x_k\in E,\,x_0<\dots<x_m\right\}.\end{multline*}
If, moreover,  $E$ is  open bounded then $|\mu_f|(E)$ coincides with   the \emph{variation} $V(f, E)$ of $f$ on $E$ \cite{FP}, given by
\[V(f, E):=\sup\left\{\int_Ef(x)\phi'(x)\,dx:\, \phi\in C^1_{\text{c}}(E),\,|\phi|\le1\right\}.\]
 If $f:\R\to\R$ is locally BV it is convenient to introduce the function
\[\rho_f(x):=|f(x^+)-f(x)|+|f(x)-f(x^-)|-|f(x^+)-f(x^-)|\,\quad\forall x\in\mathbb R.\]
Notice that $\rho_f(x)$ equals twice the distance from $f(x)$ to the interval whose end-points are $f(x^-), f(x^+)$.

\noindent
Here is  how the pointwise variation of a BV function on a bounded \emph{open} interval is related to its variation.
\begin{proposition}\label{prop:pVlambda} Let $f:\R\to\R$ be locally of bounded variation.
Then for every bounded open interval $E$:
\[\label{tag:pvv}\pV(f,E)=|\mu_f|(E)+\sum_{x\in E}\rho_f(x).
\]
\end{proposition}
\begin{proof} Given $\varepsilon>0$ we can find $x_0<x_1<\dots< x_m$ in $E$ such that
\[\pV(f,E)-\varepsilon<\sum_{k=1}^m|f(x_k)-f(x_{k-1})|;\]
now for every $k\in\{0,\dots,m\}$ we pick $x'_k,\,x''_k\in E$ such that
\[x'_0<x_0;\,x_m<x''_m; \quad x_{k-1}<x''_{k-1}<x'_k<x_k\]
for every $k=1,\dots,m$. Consider now the set $\{x'_k,x_k,x''_k:\, k=0,\dots,m\}$;
by the triangular inequality we get
\[\begin{aligned}\pV(f,E)-\varepsilon&<\displaystyle\sum_{k=1}^m|f(x_k)-f(x_{k-1})|\\
&\le\sum_{k=0}^m(|f(x_k)-f(x'_k)|+|f(x''_k)-f(x_k)|)+\sum_{k=1}^m|f(x'_k)-f(x''_{k-1})|\\&\le\pV(f,E);\end{aligned}\]
taking limits in the preceding inequality as $x'_k$ increases to $x_k$ and $x''_k$ decreases to $x_k$ we get
\[\begin{aligned}\pV(f,E)-\varepsilon&<\displaystyle\sum_{k=0}^m(|f(x_k)-f(x^-_k)|+|f(x^+_k)-f(x_k)|)+\sum_{k=1}^m|f(x^-_k)-f(x^+_{k-1})|\\
&\le\pV(f,E),\end{aligned}\]
which immediately yields
\[\begin{aligned}\pV(f,E)-\varepsilon&<\displaystyle\left(\sum_{k=1}^m|f(x^-_k)-f(x^+_{k-1})|+\sum_{k=0}^m|f(x_k^+)-f(x_k^-)|\right)+\sum_{k=0}^m\rho_f(x_k)\\
&\le\pV(f,E);\end{aligned}\] taking suprema on $\{x_0,\dots,x_m\}$ this easily gives
\[\pV(f,E)-\varepsilon<|\mu_f|(E)+\dsum_{x\in E}\rho_f(x)\le\pV(f,E),\]
 and ends the proof.
\end{proof}
\begin{remark}\label{nota} Notice that the claim of Proposition~\ref{prop:pVlambda} does not hold, in general, if $E$ is not open.
It is easy to see that for a {\em compact} interval $[a,b]$ ($a<b$) we have
\[\begin{aligned}\pV(f,[a,b])&=\pV(f,]a,b[)+|f(a)-f(a^+)|+|f(b)-f(b^-)|\\
&=|\mu_f|(]a,b[)+\sum_{x\in]a,b[}\rho_f(x)+|f(a)-f(a^+)|+|f(b)-f(b^-)|.\end{aligned}\]
This proves actually that $\pV(f,I)$ and $|\mu_f|(I)$ coincide for every bounded interval $I$ if and only if $f$ is continuous; thus $\pV(f,I)$ gives rise to a measure if and only if $f$ is continuous.
\end{remark}

\subsection{The Euler-Mac Laurin formula}
Let $f\in\ope{BV}_{\rm loc}(\R)$. The {\em mid-value modification $f_m$} for $f$ is the function defined by
\[f_{m}(x):=\dfrac{f(x^-)+f(x^+)}2.\]
The following version of the integration by parts formula for BV functions will be used in the sequel.
\begin{lemma}[Integration by parts for BV functions]
If $f,\,g:\mathbb R\to\R$ are locally of bounded variation then, for every $a<b$:
\begin{equation}\label{tag:parts}\int_{[a,b[}g_m(x)\,d\mu_f(x)=g(b^-)f(b^-)-g(a^-)f(a^-)-
\int_{[a,b[}f_m(x)\,d\mu_g(x).\end{equation}
\end{lemma}
\begin{proof} By following the lines of the proof of \cite[Theorem 3.36]{F} one gets
\[\int_{[a,b[}g(x^-)\,d\mu_f(x)=g(b^-)f(b^-)-g(a^-)f(a^-)-
\int_{[a,b[}f(x^+)\,d\mu_g(x),\]
\[\int_{[a,b[}g(x^+)\,d\mu_f(x)=g(b^-)f(b^-)-g(a^-)f(a^-)-
\int_{[a,b[}f(x^-)\,d\mu_g(x).\]
The result is obtained by summing up term by term the members of the above equalities, and dividing by 2.
\end{proof}
The following Euler-Maclaurin  formula for the sums $\displaystyle\sum_{a\le k<b}f_m(k)$ holds: differently from the classical one, the sums involve the mid-value modification of $f$, due to its  possible discontinuities.
The first Bernoulli polynomial $B_1(x)=x-\dfrac12$, restricted to $[0,1]$, is involved in the first-order Euler-Maclaurin formula for smooth functions \cite[Theorem 12.27]{MT}; we will use  here the mid-value modification of its extension by periodicity $\beta_1:\R\to\R$ defined by
\[\beta_1(x):=\begin{cases}B_1(x-[x])&\text{ if }x\notin \Z,\\
0&\text{ otherwise}.
\end{cases}\]
\begin{theorem}[First-order Euler-Maclaurin formula for BV functions]\label{thm:emmid} Assume that $f:\R\to\R$ is locally of bounded variation. Then, for any $a<b$ in $\mathbb Z$,
\begin{equation}\label{tag:nuovo}\sum_{a\le k<b}f_m(k)=\int_a^bf(x)\,dx-\dfrac12(f(b^-)-f(a^-))+\int_{]a,b[}\beta_1(x)\,d\mu_f(x).
\end{equation}
\end{theorem}
\begin{proof} The proof of Theorem~\ref{thm:emmid} goes formally as that of the classical first-order Euler-Maclaurin formula.
Clearly $\b_1$ is locally of bounded variation;
plainly $\mu_{\beta_1}=\lambda_1-\displaystyle\sum_{n\in\Z}\de_{n}$, where $\lambda_1$ is the Lebesgue measure. Since $(\beta_1)_m=\beta_1$,
applying formula (\ref{tag:parts}) with $g=\b_1$  we get
\[\begin{aligned}\int_{[a,b[}\b_1(x)\,d\mu_f(x)&=\b_1(b^-)\,f(b^-)-\b_1(a^-)\,f(a^-)-\int_{[a,b]}f_m(x)\,d\mu_{\beta_1}(x)\\
&=\dfrac{f(b^-)-f(a^-)}2-\int_a^bf(x)\,dx+\int_{[a,b[}f_m(x)\,d\left(\displaystyle\sum_{k\in\Z}\de_{k}\right)(x)\\
&=\dfrac{f(b^-)-f(a^-)}2-\int_a^bf(x)\,dx+\sum_{a\le k<b}f_m(k),\end{aligned}\]
which we can rewrite
\[\sum_{a\le k<b}f_m(k)=\int_a^bf(x)\,dx-\dfrac{f(b^-)-f(a^-)}2+\int_{[a,b[}\beta_1(x)\,d\mu_f(x);\]
since  $\b_1(a)=0$, we get  $\displaystyle\int_{[a,b[}\beta_1(x)\,d\mu_f(x)=\int_{]a,b[}\beta_1(x)\,d\mu_f(x)$.
\end{proof}
Theorem~\ref{thm:emmid} yields an alternative proof of (\ref{tag:critint_new}).
\begin{proof}[Alternative proof of Theorem~\ref{remark:critint4}.]
To deduce \eqref{tag:critint_new} from the preceding theorem we rewrite
\begin{equation}\notag\sum_{a\le k<b}f(k)=\int_a^bf(x)\,dx-\dfrac{f(b)-f(a)}2+R,\end{equation}
\[\begin{aligned}
 R:&=\int_{]a,b[}\b_1(x)\,d\mu_f(x)+\sum_{a\le k<b}f(k)-\sum_{a\le k<b}f_m(k)+\dfrac12\left[f\right]_a^b-\dfrac{f(b^-)-f(a^-)}2\\
 &=\int_{]a,b[}\b_1(x)\,d\mu_f(x)+\dfrac12\sum_{a<k<b}\big((f(k)-f(k^-))+(f(k)-f(k^+))\big)+\\
 &\phantom{AAAAAAAAAAAAAAAAAAAA}+\dfrac12\big((f(a)-f(a^+))+(f(b)-f(b^-))\big).
 \end{aligned}\]
 so that
 \begin{multline}\label{tag:B}
 |R|\le  \left|\int_{]a,b[}\b_1(x)\,d\mu_f(x)\right|+\dfrac12\sum_{a<k<b}\big(|f(k)-f(k^-)|+|f(k)-f(k^+)|\big)+\\
+\dfrac12\big(|f(a)-f(a^+)|+|f(b)-f(b^-)|\big).
 \end{multline}
 Now, since $\displaystyle\int_{]a,b[}|\b_1(x)|\,d|\mu_f|(x)$ lacks the contribution of the jumps of $f$ on the integers and $|\b_1|\le 1/2$,
 \[\begin{aligned}\int_{]a,b[}|\b_1(x)|\,d|\mu_f|(x)&\le \dfrac12|\mu_f|\big(]a,b[\setminus\mathbb Z\big)\\
 &= \dfrac12|\mu_f|\big(]a,b[\big)-\dfrac12\sum_{a< k<b}{|f(k^+)-f(k^-)|}.
 \end{aligned}\]
 It follows from \eqref{tag:B} and Proposition~\ref{prop:pVlambda}, taking account of Remark \ref{nota}, that
 \[\begin{aligned}|R|&\le \dfrac12\Big(|\mu_f|\big(]a,b[\big)+\sum_{a<k<b}\rho_f(k)\Big)+\dfrac12\big(|f(a)-f(a^+)|+|f(b)-f(b^-)|\big)\\
 &\le
 \pV(f, ]a,b[)
+\dfrac12\big(|f(b)-f(b^-)|+|f(a^+)-f(a)|\big)=\dfrac12\pV(f, [a,b]).
 \end{aligned}\]
\end{proof}

\section*{References}
\bibliographystyle{elsarticle-num}
\bibliography{bibeulermac}
\end{document}